\documentclass[a4paper,10pt]{article}
\usepackage[T1]{fontenc}
\usepackage[cp1250]{inputenc}
\usepackage[english]{babel}
\usepackage{amsmath}
\usepackage{amsfonts}
\usepackage{amssymb}
\usepackage{mathrsfs}
\usepackage{amsthm}
\usepackage{graphicx}
\usepackage{esint}

\usepackage{color}
\usepackage{euscript}
\usepackage{cite}
\usepackage{mathtools}
\usepackage{enumerate}

\newcommand{\eps}{\varepsilon}

\renewcommand{\leq}{\leqslant}
\renewcommand{\geq}{\geqslant}

\newcommand{\R}{\mathbb{R}}

\newcommand{\A}{\mathfrak{A}}

\newcommand{\eq}[1]{\begin{equation}{#1}\end{equation}}
\newcommand{\mlt}[1]{\begin{multline}{#1}\end{multline}}
\newcommand{\alg}[1]{\begin{align}{#1}\end{align}}

\newcommand{\set}[2]{\{{#1}\mid{#2}\}}
\newcommand{\Set}[2]{\Big\{{#1}\,\Big|\;{#2}\Big\}}

\newcommand{\twodots}{\,..\,}

\newcommand{\Z}{\mathbb{Z}}

\newcommand{\D}{\mathbb{D}}

\newtheorem{Le}{Lemma}

\newtheorem{St}[Le]{Proposition}

\numberwithin{equation}{section}

\begin{document}

\title{On embedding of Besov spaces of zero smoothness into Lorentz spaces}

\author{Dmitriy Stolyarov\thanks{Supported by the Russian Science Foundation grant 19-71-10023.}}

\maketitle

\medskip 

\centerline{\emph{Dedicated to professor Besov on the occasion of his anniversary}}

\medskip
%\subjclass{Primary 60G46; Secondary 46E35}

\begin{abstract}
We show that the zero smoothness Besov space~$B_{p,q}^{0,1}$ does not embed into the Lorentz space~$L_{p,q}$ unless~$p=q$; here~$p,q\in (1,\infty)$. This answers negatively a question proposed by O. V. Besov.
\end{abstract}

\section{Besov spaces of zero smoothness}\label{S1}
Besov spaces play an important role in analysis, PDE, approximation theory, and other parts of mathematics. At first glance they may be thought of as refinements of Sobolev spaces. The Besov spaces were introduced in~\cite{Besov1961} and now their description may be found in many textbooks. See, e.g.,~\cite{BIN1978},~\cite{Grafakos2009Modern}, or~\cite{Peetre1976}. In recent years there is a constant interest in generalization and specification of this already fine scale. For Besov--Lorentz spaces, see~\cite{HSR2022},~\cite{Peetre1976},~\cite{SeegerTrebels2019}, and~\cite{Stolyarov2022}. For so-called logarithmic Besov spaces, see~\cite{CGO2008}, \cite{CGO2011}, \cite{CobosDominguez2015bis},~\cite{CobosDominguez2015}, and~\cite{DominguezTikhonov2018}. The latter class of spaces is often named as Besov spaces of zero smoothness. For Besov spaces of generalized smoothness, see, e.g.,~\cite{Goldman1987},~\cite{Goldman2007},~\cite{Netrusov1987}, the literature on this topic is vast. Yet another version of this notion was suggested by O. V. Besov in~\cite{Besov2012} and~\cite{Besov2014}. The definition reads as follows. Let~$p,q,r \in [1,\infty)$ be three parameters and let~$f$ be a locally summable function on~$\R^d$. Define the seminorm
\eq{\label{BesovNorm}
\|f\|_{B_{p,q}^{0,r}} = \Big(\int\limits_0^\infty \Big(\int\limits_{\R^d}\big(\delta_r [f;h](x)\big)^p\,dx\Big)^{\frac{q}{p}}\,\frac{dh}{h}\Big)^{\frac1q},
}
where~$\delta_r [f;h]\colon \R^d\to \R_+$ is the 'averaged modulus of smoothness',
\eq{
\delta_r[f,h](x) = \Big(\fint\limits_{[-h,h]^d}\fint\limits_{[-h,h]^d}|f(x+y) - f(x+z)|^r\,dy\,dz\Big)^{\frac1r}.
}
It was proved in~\cite{Besov2014} that~$\delta_r[f;h]$ is equivalent to
\eq{\label{equivalentdelta}
\Big(\fint\limits_{[-h,h]^d}\Big|f(x+y) - \fint\limits_{[-h,h]^d}f(x+z)\,dz\Big|^r\,dy\Big)^{1/r}.
} 
Let~$B_{p,q}^{0,r}$ be the completion of the set of smooth compactly supported functions in the corresponding seminorms. The spaces~$B_{p,q}^{0,r}$ were used in~\cite{Besov2012} and~\cite{Besov2014} obtain sharp embedding theorems. Note that the definitions in those two papers differ from each other a little bit, and our definition is also slightly different. These variations do not make sense for the questions we are going to investigate.

Recall the definition of the Lorentz quasi-norm:
\eq{\label{LorentzDef}
\|f\|_{L_{p,q}} = p^{\frac{1}{q}} \Big\|t |\set{x\in \R^d}{|f(x)| \geq t}|^{\frac{1}{p}}\Big\|_{L_q(\R_+,\frac{dt}{t})},
}
where the absolute value of a set is its Lebesgue measure. In~\cite{Besov2014}, O. V. Besov asked whether the space~$B_{p,q}^{0,1}$ embeds continuously into the Lorentz space~$L_{p,q}$. He mentioned that A. I. Tyulenev proved the said embedding in the case~$1< p=q < \infty$. We are ready to state our main results. From now on we limit our considerations to the case only~$r=1$; the study of other cases is also interesting. We will consider the cases~$q < p$ and~$p < q$ separately since the reasonings are different in these cases.
\begin{St}\label{Prop1}
Let~$1 \leq q < p <\infty$. Then the space~$B_{p,q}^{0,1}$ does not embed continuously into~$L_{p,q}$.
\end{St} 
\begin{St}\label{Prop2}
Let~$1 < p < q <\infty$. Then the space~$B_{p,q}^{0,1}$ does not embed continuously into~$L_{p,q}$.
\end{St} 
Theorem~$2$ in~\cite{Besov2012} says that~$B_{p,q}^{0,1}$ embeds continuously into the classical Besov space~$B_{p,q}^0$ (the latter space is defined via Fourier analytical decompositions as in~\cite{Peetre1976}). It is known that~$B_{p,q}^0$ does not embed into~$L_{p,q}$. The propositions above seem to be more difficult than this folklore fact.

The plan of the paper is: we finish this section with a useful discretization of the seminorm~\eqref{BesovNorm} and some heuristics behind the proofs of Propositions~\ref{Prop1} and~\ref{Prop2}; Sections~\ref{S2} and~\ref{S3} contain the proofs of these propositions; the final Section~\ref{S4} is devoted to Tyulenev's reasoning.

We will be considering only cubes with edges parallel to the coordinate axis. For a cube~$Q$, let~$c_Q$ be its center and let~$\ell(Q)$ be its sidelength. Let~$\D_n$ be the set of dyadic cubes of generation~$n$:
\eq{
\D_n = \Set{\prod\limits_{j=1}^d[2^{-n}k_j, 2^{-n}(k_j+1))}{k\in \Z^d}.
}
\begin{Le}\label{equivalentnorm}
The seminorms
\eq{\label{equivalentnormformula}
\|f\|_* =  \Big(\sum\limits_{k\in \Z} \Big(\sum\limits_{Q\in \D_k} |Q| \big(\delta[f;0.6\ell(Q)](c_Q)\big)^p\Big)^{\frac{q}{p}}\Big)^{1/q}
}
and~\eqref{BesovNorm} are equivalent.
\end{Le}
The number~$0.6$ in~\eqref{equivalentnormformula} may be replaced with any number larger than~$0.5$ (the cubes over which we compute oscillations should overlap). It is convenient for further reasonings that this number is smaller than~$1$. Lemma~\ref{equivalentnorm} is standard and we omit its proof.

\section{Case $q < p$}\label{S2}
Let~$\varphi$ be a smooth function with zero integral supported in the unit cube~$[-1/2,1/2]^d$. Let~$N$ be a large natural number. We equip the set~$[1\twodots 2^{dN}]$ with the counting measure and consider Lorentz spaces~$\ell_{p,q}^{2^{dN}}$ of finite sequences. Here and in what follows, the notation~$[a\twodots b]$ designates the integer interval~$\set{x\in \Z}{a\leq x \leq b}$. Let~$\A^N = \{a_j\}_{j=1}^{2^{dN}}$ be a sequence of positive numbers such that
\eq{
\|\A^N\|_{\ell_{p,q}^{2^{dN}}} \gtrsim N^{1/q-1/p}\|\A^N\|_{\ell_p^{2^{dN}}}.
}
Here and in what follows, the notation~$A \lesssim B$ means there exists a constant~$C$ independent of certain parameter such that~$A \leq CB$; the parameter and the nature of uniformity is usually clear from the context. For example, the constant should not depend on~$N$ in the above inequality. We will also write~$A \asymp B$ in the case~$A\lesssim B$ and~$B \lesssim A$. As an example of the sequence~$\A^N$, we may consider 
\eq{
a_j = 2^{-dn/p},\quad \text{when}\quad j\in [2^{dn}\twodots 2^{d(n+1)}),
}
here~$n=0,1,\ldots, N-1$. 

The adjustment of~$\varphi$ to a cube~$Q$ is the function~$\varphi_Q$ given by the formula
\eq{
\varphi_Q(x) = \varphi\Big(\frac{x-c_Q}{\ell(Q)}\Big),\qquad x\in \R^d.
}
Let~$n$ be another natural number that is much larger than~$N$ (the precise dependence will be provided later). Let~$Q_{1}, Q_2,\ldots, Q_{2^{dN}}$ be cubes in~$\D_N$ lying inside~$[0,1]^d$. Let~$C_j$ be cubes in~$\D_n$ such that for each~$j \in [1\twodots 2^{dN}]$ the cube~$C_j$ contains~$c_{Q_j}$. We are ready to define the function~$f_{N,n}$ that will prove Proposition~\ref{Prop1}:
\eq{\label{FormulaForf}
f_{N,n} = 2^{\frac{dn}{p}}\sum\limits_{j=1}^{2^{dN}}a_j\varphi_{C_j}.
}
The simple identity
\eq{
\|f_{N,n}\|_{L_{p,q}} \asymp \|\A^N\|_{\ell_{p,q}^{2^{dN}}}
}
implies Proposition~\ref{Prop1} will be proved once we obtain the following lemma.
\begin{Le}
If~$1\leq q < p < \infty$, then
\eq{\label{LemmaInequality}
\varlimsup\limits_{n\to \infty}\|f_{N,n}\|_{B_{p,q}^{0,1}} \lesssim \|\A^N\|_{\ell_p^{2^{dN}}}.
}
\end{Le}
\begin{proof}
We will be writing~$f$ for~$f_{N,n}$ for simplicity. To prove the desired bound, we use Lemma~\ref{equivalentnorm}, raise the left hand side of~\eqref{LemmaInequality} to the power~$q$ and arrive at
\eq{
\sum\limits_{k\in \Z} \Big(\sum\limits_{Q\in \D_k} |Q| \big(\delta[f;0.6\ell(Q)](c_Q)\big)^p\Big)^{\frac{q}{p}}
}
We will estimate the parts of the sum where~$k < 0$,~$k\in [0\twodots  N-1]$,~$k \in [N\twodots n]$, and~$k > n$ individually.

\paragraph{Estimate for~$k <0$.} 
Fix~$k$ and start with the estimate
\mlt{\label{eq14}
\delta[f;0.6\ell(Q)](c_Q)\lesssim \fint\limits_{1.2Q}\Big|f(x) - \fint\limits_{1.2Q} f(y)\,dy\Big|\,dx\\ \leq 2\fint\limits_{1.2Q} |f(x)|\,dx \lesssim 2^{dk}\|f\|_{L_1} \leq 2^{dk}\|f\|_{L_p} \lesssim 2^{dk}\|\A^N\|_{\ell_p^{2^{dN}}},\qquad Q \in \D_k.
}
We note that at for at most~$3^d$ dyadic cubes~$Q$ the support of~$f$ intersects~$1.2Q$. So, we obtain the bound
\eq{\label{eq15}
\sum\limits_{Q\in \D_k} |Q| \big(\delta[f;0.6\ell(Q)](c_Q)\big)^p \lesssim 2^{d(p-1)k}\|\A^N\|_{\ell_p^{2^{dN}}}^p.
}
Therefore,
\eq{\label{eq16}
\sum\limits_{k\leq 0} \Big(\sum\limits_{Q\in \D_k} |Q| \big(\delta[f;0.6\ell(Q)](c_Q)\big)^p\Big)^{\frac{q}{p}}\lesssim \sum\limits_{k\leq 0}2^{\frac{dq(p-1)}{p}k}\|\A^N\|_{\ell_p^{2^{dN}}}^q \lesssim \|\A^N\|_{\ell_p^{2^{dN}}}^q
}
since~$p > 1$.

\paragraph{Estimate for~$k\in [0\twodots  N-1]$.} Here we start with the estimate
\eq{
\sup\limits_{Q \in \D_N}\fint\limits_Q|f(x)| \lesssim 2^{dN}d^{-\frac{d(p-1)}{p}n}\|\A^N\|_{\ell_\infty^{2^{dN}}},
}
which follows immediately from~\eqref{FormulaForf} and the fact that each cube~$Q \in \D_N$ contains at most one~$C_j$. Since averaging does not increase the~$L_\infty$ norm of a function,
we obtain
\eq{
\sup\limits_{Q\in \D_k} \fint\limits_Q|f(x)| \lesssim 2^{dN}d^{-\frac{d(p-1)}{p}n}\|\A^N\|_{\ell_\infty^{2^{dN}}}
}
for any~$k \leq N$. Therefore,
\eq{\label{eq19}
\sum\limits_{Q\in \D_k} |Q| \big(\delta[f;0.6\ell(Q)](c_Q)\big)^p \lesssim 2^{dpN}d^{-d(p-1)n}\|\A^N\|_{\ell_\infty^{2^{dN}}}^p,\qquad k \leq N,
}
and
\eq{\label{eq222}
\sum\limits_{0 \leq k < N} \Big(\sum\limits_{Q\in \D_k} |Q| \big(\delta[f;0.6\ell(Q)](c_Q)\big)^p\Big)^{\frac{q}{p}} \lesssim N 2^{dqN}d^{-\frac{dq(p-1)}{p}n}\|\A^N\|_{\ell_\infty^{2^{dN}}}^q.
}
This quantity tends to zero when~$n \to \infty$ and~$N$ is fixed.

\paragraph{Estimate for~$k\in [N\twodots n]$.} It suffices to prove the bound
\eq{\label{IndScaleIntCubes}
\Big(\sum\limits_{Q\in \D_k} |Q|\big(\delta[f;0.6\ell(Q)](c_Q)\big)^p\Big)^{1/p}\lesssim 2^{\frac{d(p-1)}{p}(k-n)}\|\A^N\|_{\ell_p^{2^{dN}}},\qquad k\in [N\twodots n].
}
Fix~$k$ and denote by~$K_j$ a cube in~$\D_k$ such that~$1.2 K_j$ contains~$C_j$. Note that there are at most~$3^d$ such cubes for any fixed~$j$ and they are different for different~$j$ (since the cubes~$C_j$ contain the centers of different dyadic cubes of generation~$N$). Then, by~\eqref{equivalentdelta},
\eq{
\delta[f;0.6\ell(K_j)](c_{K_j}) \leq 2\fint\limits_{1.2K_j}|f(x)|\,dx = 2^{dk}\int\limits_{C_j}|f(x)|\,dx \lesssim 2^{d(k-n)+\frac{d}{p}n}a_j.
}
Using the information about the cubes~$K_j$ given above, we estimate
\eq{
\Big(\sum\limits_{Q\in \D_k} |Q|\big(\delta[f;0.6\ell(Q)](c_Q)\big)^p\Big)^{1/p} \lesssim \Big(\sum\limits_{j=1}^{2^{dN}}2^{dp(k-n)+dn-dk} a_j^p\Big)^{1/p},
}
which leads to~\eqref{IndScaleIntCubes}. Thus,
\eq{\label{eq218}
\sum\limits_{N \leq k < n} \Big(\sum\limits_{Q\in \D_k} |Q| \big(\delta[f;0.6\ell(Q)](c_Q)\big)^p\Big)^{\frac{q}{p}} \lesssim \sum\limits_{k \leq n}2^{\frac{dq(p-1)}{p}(k-n)}\|\A^N\|_{\ell_p^{2^{dN}}}^q \lesssim \|\A^N\|_{\ell_p^{2^{dN}}}^q.
}

\paragraph{Estimate for~$k > n$.} Fix~$k$. In this case,
\mlt{\label{eq25}
\Big(\sum\limits_{Q\in \D_k} |Q|\big(\delta[f;0.6\ell(Q)](c_Q)\big)^p\Big)^{1/p} \lesssim \|\A^N\|_{\ell_p^{2^{dN}}}\Big(\sum\limits_{Q\in \D_k} |Q|\big(\delta[2^{\frac{d}{p}n}\varphi_{C_1};0.6\ell(Q)](c_Q)\big)^p\Big)^{1/p}\\ 
= \|\A^N\|_{\ell_p^{2^{dN}}}\Big(\sum\limits_{Q\in \D_{k-n}} |Q|\big(\delta[\varphi;0.6\ell(Q)](c_Q)\big)^p\Big)^{1/p}
}
by dilation. It remains to prove that
\eq{\label{eq26}
\Big(\sum\limits_{Q\in \D_{m}} |Q|\big(\delta[\varphi;0.6\ell(Q)](c_Q)\big)^p\Big)^{1/p} \lesssim 2^{-m},\qquad m \geq 0,
}
which follows from the estimate
\eq{
\Big\|\delta[\varphi;0.6\cdot 2^{-m}](\cdot)\Big\|_{L_\infty} \lesssim 2^{-m}\|\nabla \varphi\|_{L_\infty}.
}
Therefore,
\eq{
\Big(\sum\limits_{Q\in \D_k} |Q|\big(\delta[f;0.6\ell(Q)](c_Q)\big)^p\Big)^{1/p} \lesssim 2^{-(k-n)} \|\A^N\|_{\ell_p^{2^{dN}}},\qquad k > n,
}
and
\eq{\label{eq224}
\sum\limits_{k > n} \Big(\sum\limits_{Q\in \D_k} |Q| \big(\delta[f;0.6\ell(Q)](c_Q)\big)^p\Big)^{\frac{q}{p}}\lesssim \sum\limits_{k > n}2^{-(k-n)q} \|\A^N\|_{\ell_p^{2^{dN}}}^q \lesssim \|\A^N\|_{\ell_p^{2^{dN}}}^q.
}

\paragraph{Conclusion.} We collect the estimates~\eqref{eq16},~\eqref{eq222},~\eqref{eq218}, and~\eqref{eq224}. We choose~$n$ so large that the right hand side of~\eqref{eq222} does not exceed the right hand side of~\eqref{LemmaInequality} and complete the proof of the lemma.
\end{proof}

\section{Case $q > p$}\label{S3}
Let us call functions of the type~$\sum_{j}\varphi_{Q_j}$, where~$\sum_j |Q_j| = 1$ and the cubes~$Q_j$ are dyadic, disjoint, and have the same sidelength, the \emph{building blocks}. To prove Proposition~\ref{Prop2}, it suffices, given any~$T$, to construct the function~$\Phi = \sum_{k=1}^T \Phi_k$, where the~$\Phi_k$ are building blocks with disjoint supports, and~$\|\Phi\|_* \lesssim T^{1/q}$.
Then, by Lemma~\ref{equivalentnorm},~$\|\Phi\|_{B_{p,q}^{0,1}}\lesssim T^{1/q}$, whereas~$\|\Phi\|_{L_{p,q}}\gtrsim T^{1/p}$ (since all the building blocks have one and the same distribution function). This proves Proposition~\ref{Prop2} since~$T$ may be chosen arbitrarily large. 

The function~$\Phi$ may be easily constructed with the consecutive application of the following lemma.
\begin{Le}\label{InductionStep}
There exists a constant~$C$ such that for any smooth compactly supported function~$f$ and any~$\eps > 0$, there exists a building block~$\Psi$ such that
\eq{
\|f + \Psi\|_* \leq (1+\eps)\big(\|f\|_* + C\big)^{1/q}
}
and the supports of~$f$ and~$\Psi$ are disjoint.
\end{Le}
We start with a simple observation. Let~$a = \{a_k\}_{k\in \Z}$ and~$b = \{b_k\}_{k\in\Z}$ be two~$\ell_q$ sequences. If the sequences~$a$ and~$b$ have disjoint supports, then~$\|a+b\|_{\ell_q}^q = \|a\|_{\ell_q}^q + \|b\|_{\ell_q}^q$. We claim that if the sequences~$a$ and~$b$ are almost disjointly supported, then this identity almost holds. We omit the proof of the following lemma.
\begin{Le}\label{EllqDisjoint}
Let~$a$ and~$b$ be two~$\ell_q$ sequences, where~$q\in [1,\infty)$ is a fixed parameter. For any~$\eps > 0$ there exists~$\delta > 0$ such that if 
\alg{
&\sum\limits_{k\in I}|a_k|^q \geq (1-\delta)\|a\|_{\ell_q}^q;\\
&\sum\limits_{k\in J}|b_k|^q \geq (1-\delta)\|b\|_{\ell_q}^q
}
for some disjoint sets~$I,J \subset \Z$, then
\eq{
\|a+b\|_{\ell_q}^q \leq (1+\eps)\Big(\|a\|_{\ell_q}^q + \|b\|_{\ell_q}^q\Big).
}
\end{Le}
When proving Lemma~\ref{InductionStep}, we may assume~$f$ is a smooth function supported in the unit cube. For such functions, we have the estimates
\alg{
\label{SmoothBounds1}\Big(\sum\limits_{Q\in \D_k}|Q|\big(\delta[f;0.6\ell(Q)](c_Q)\big)^p\Big)^{q/p} \lesssim 2^{\frac{dq(p-1)}{p}k}\|f\|_{L_p}^q,\qquad k \leq 0;\\
\label{SmoothBounds2}\Big(\sum\limits_{Q\in \D_k}|Q|\big(\delta[f;0.6\ell(Q)](c_Q)\big)^p\Big)^{q/p} \lesssim 2^{-qk}\|\nabla f\|_{L_\infty}^q,\qquad k \geq 0,
}
which are obtained in the same way as~\eqref{eq16} and~\eqref{eq26}, respectively.
\begin{proof}[Proof of Lemma~\ref{InductionStep}.] Assume~$f$ is supported in the unit cube. By Lemma~\ref{EllqDisjoint} and the bounds~\eqref{SmoothBounds1},~\eqref{SmoothBounds2}, it suffices, given some large~$M \in \mathbb{N}$, to construct a building block~$\Phi$ such that~$\|\Phi\|_* \gtrsim 1$, and
\eq{\label{Concentration}
\Big(\sum\limits_{Q\in \D_l}|Q|\big(\delta[\Phi;0.6\ell(Q)](c_Q)\big)^p\Big)^{q/p}\lesssim 2^{-\nu |l-M|},\qquad l\in \Z,
}
where~$\nu$ is a fixed positive number independent of~$M$ (the multiplicative constants in the inequalities are also independent of~$M$). To do this, we pick some~$N$,~$n \geq 3N$, and construct the function~$\tilde{\Psi}$ similar to~\eqref{FormulaForf}, i.e.
\eq{
\tilde{\Psi} = 2^{\frac{dn}{p}}\sum\limits_{j=1}^{2^{dN}}2^{-\frac{dN}{p}}\varphi_{C_j}.
}
In other words, we plug~$\A^N = \{2^{-dN/p}\}_{j=1}^{2^{dN}}$ into~\eqref{FormulaForf}. Note that~$\|\A^N\|_{\ell_p^{2^{dN}}}=1$. The function~$\Psi$ is then defined as the~$2^{n-N}$-dilation of~$\tilde{\Psi}$ with the preservation of the~$L_p$ norm, i.e.
\eq{
\Psi = \sum\limits_{j=1}^{2^{dN}}\varphi_{2^{n-N}C_j}.
}
This is a building block indeed since~$\sum_{j=1}^{2^{dN}}|2^{n-N}C_j| = 1$. It is easy to see that
\eq{
\Big(\sum\limits_{Q\in \D_{N}}|Q|\big(\delta[\Psi;0.6\ell(Q)](c_Q)\big)^p\Big)^{q/p} \gtrsim 1.
}
Thus,~$\|\Psi\|_* \gtrsim 1$, and it remains to prove~\eqref{Concentration}. Set~$N=M$. We will be using the bounds obtained in Section~\ref{S2}. We will analyze the inequality~\eqref{Concentration} in the cases~$l\leq N-n$,~$l\in [N-n\twodots 2N-n]$,~$l\in [2N-n\twodots N]$, and~$l \geq N$ individually.

\paragraph{Case~$l \leq N-n$.} We refine the estimate~\eqref{eq14}, calculating the~$L_1$-norm directly:
\mlt{
\fint\limits_{1.2Q}\Big|\tilde{\Psi}(x) - \fint\limits_{1.2Q} \tilde{\Psi}(y)\,dy\Big|\,dx\\ \leq 2\fint\limits_{1.2Q} |\tilde{\Psi}(x)|\,dx \lesssim 2^{dk}\|\tilde{\Psi}\|_{L_1} \lesssim 2^{dk - \frac{p-1}{p}(n-N)},\qquad Q\in \D_k, k \leq 0.
}
Thus,~\eqref{eq15} is improved to
\eq{
\Big(\sum\limits_{Q\in \D_k} |Q|\big(\delta[\tilde{\Psi};0.6\ell(Q)](c_Q)\big)^p\Big)^{q/p} \lesssim 2^{\frac{dq(p-1)}{p}(k-n+N)},\qquad k \leq 0.
}
Therefore,
\eq{
\Big(\sum\limits_{Q\in \D_l}|Q|\big(\delta[\Psi;0.6\ell(Q)](c_Q)\big)^p\Big)^{q/p} \lesssim 2^{\frac{dq(p-1)}{p}l},\qquad l\leq N-n.
}
Note that this is bounded by~$2^{\frac12 dq(1-1/p)l}$, provided~$n \geq 2N$, and we have verified~\eqref{Concentration} in this case.

\paragraph{Case~$l\in [N-n\twodots 2N-n]$.} We use~\eqref{eq19}:
\mlt{
\Big(\sum\limits_{Q\in \D_l}|Q|\big(\delta[\Psi;0.6\ell(Q)](c_Q)\big)^p\Big)^{q/p}  = \Big(\sum\limits_{Q\in \D_{l-N+n}}|Q|\big(\delta[\tilde{\Psi};0.6\ell(Q)](c_Q)\big)^p\Big)^{q/p}\\
 \lesssim 2^{dqN}2^{-\frac{dq(p-1)}{p}n}\|\A^N\|_{\ell_\infty^{2^{dN}}}^q = 2^{\frac{dq(p-1)}{p}(N-n)}, \quad l\in [N-n\twodots 2N-n].
}
If~$2N < n$, then,
\eq{
\Big(\sum\limits_{Q\in \D_l}|Q|\big(\delta[\Psi;0.6\ell(Q)](c_Q)\big)^p\Big)^{q/p}  \lesssim 2^{\frac{dq(p-1)}{2p}(l-N)},\quad l\in [N-n\twodots 2N-n].
}

\paragraph{Case~$l\in [2N-n\twodots N]$.} We use~\eqref{IndScaleIntCubes}:
\mlt{
\Big(\sum\limits_{Q\in \D_l}|Q|\big(\delta[\Psi;0.6\ell(Q)](c_Q)\big)^p\Big)^{q/p} \\ = \Big(\sum\limits_{Q\in \D_{l-N+n}}|Q|\big(\delta[\tilde{\Psi};0.6\ell(Q)](c_Q)\big)^p\Big)^{q/p} \lesssim 2^{\frac{dpq}{p-1}(l - N)},\quad l\in [2N-n\twodots N]. 
}

\paragraph{Case~$l \geq N$.} The bounds \eqref{eq25} and~\eqref{eq26} imply
\eq{
\Big(\sum\limits_{Q\in \D_l}|Q|\big(\delta[\Psi;0.6\ell(Q)](c_Q)\big)^p\Big)^{q/p}  = \Big(\sum\limits_{Q\in \D_{l-N+n}}|Q|\big(\delta[\tilde{\Psi};0.6\ell(Q)](c_Q)\big)^p\Big)^{q/p} \lesssim 2^{-(l-N)q},\quad l \geq N.
}

We have finally verified~\eqref{Concentration} in all the cases and have completed the proof.
\end{proof}
\section{Case~$p=q$}\label{S4}
The material of this section is published with the permission of A. I. Tyulenev.
In the case~$p=q$ the~$B_{p,q}^{0,1}$ seminorm equals
\eq{
\Big(\int\limits_0^\infty \int\limits_{\R^d}\big(\delta_r [f;h](x)\big)^p\,dx\,\frac{dh}{h}\Big)^{\frac1p}.
}
Recall the sharp maximal function~$f^\#$ (see~\cite{Stein1993}) and note that
\eq{
\big(f^\#(x)\big)^p \lesssim \sup\limits_{h > 0} \big(\delta[f;h](x)\big)^p \lesssim \int\limits_0^{\infty}\big(\delta_r [f;h](x)\big)^p\frac{dh}{h}.
}
Therefore, if~$p > 1$, and~$f$ is a bounded compactly supported function, then
\eq{
\|f\|_{L_p} \lesssim \|f^{\#}\|_{L_p} \lesssim \|f\|_{B_{p,p}^{0,1}}.
}
This proves the embedding~$B_{p,p}^{0,1}\hookrightarrow L_p$ when~$p\in (1,\infty)$.

\bibliography{mybib}{}
\bibliographystyle{amsplain}

Dmitriy Stolyarov

St. Petersburg State University, 14th line 29B, Vasilyevsky Island, St. Petersburg, Russia.

d.m.stolyarov@spbu.ru.

\end{document}